\documentclass[12pt]{article}
\usepackage{geometry}
\usepackage[ansinew]{inputenc}
\usepackage{comment}
\usepackage{amssymb}
\usepackage{latexsym}
\usepackage{graphics}
 \usepackage{tikz}
\usepackage{graphicx}
\usepackage{tikz-3dplot}
\usepackage{epstopdf}
\usepackage{xcolor}
\usepackage{fancybox}
\usepackage{amsmath,amsfonts}
 \usepackage[pagewise]{lineno}
 \usepackage{hyperref}
\begin{document}

\setlength{\parskip}{0.3\baselineskip}

\newtheorem{theorem}{Theorem}
\newtheorem{corollary}[theorem]{Corollary}
\newtheorem{lemma}[theorem]{Lemma}
\newtheorem{proposition}[theorem]{Proposition}
\newtheorem{definition}[theorem]{Definition}
\newtheorem{remark}[theorem]{Remark}
\renewcommand{\thefootnote}{\alph{footnote}}
\newenvironment{proof}{\smallskip \noindent{\bf Proof}: }{\hfill $\Box$\hspace{1in} \medskip \\ }


\newcommand{\beqaa}{\begin{eqnarray}}
\newcommand{\eeqaa}{\end{eqnarray}}
\newcommand{\beqae}{\begin{eqnarray*}}
\newcommand{\eeqae}{\end{eqnarray*}}


\newcommand{\sii}{\Leftrightarrow}
\newcommand{\imer}{\hookrightarrow}
\newcommand{\imerc}{\stackrel{c}{\hookrightarrow}}
\newcommand{\Con}{\longrightarrow}
\newcommand{\con}{\rightarrow}
\newcommand{\conf}{\rightharpoonup}
\newcommand{\confe}{\stackrel{*}{\rightharpoonup}}
\newcommand{\pbrack}[1]{\left( {#1} \right)}
\newcommand{\sbrack}[1]{\left[ {#1} \right]}
\newcommand{\key}[1]{\left\{ {#1} \right\}}
\newcommand{\dual}[2]{\langle{#1},{#2}\rangle}
\newcommand{\intO}[1]{\int_{\Omega}{#1}\, dx}


\newcommand{\R}{{\mathbb R}}
\newcommand{\N}{{\mathbb N}}

\newcommand{\cred}[1]{\textcolor{red}{#1}}


\title{\bf Regularity for the Timoshenko system  with fractional damping}
\author{Fredy Maglorio Sobrado  Su\'arez\thanks{Corresponding author.}\\
\small Department of Mathematics, Federal University of Technology  of Paran\'a, Brazil
}

\date{}
\maketitle

\let\thefootnote\relax\footnote{{\it Email address:} {\rm fredy@utfpr.edu.br} (Fredy M. Sobrado Suárez). Published on ago 28, 2023. In Journal of Engineering Research (ISSN 2764-1317). DOI: 10.22533/at.ed.3173292324083}
\begin{abstract}	
We study the Regularity of the Timoshenko system with two fractional dampings $(-\Delta)^\tau u_t$ and $(-\Delta)^\sigma \psi_t$; both of the parameters $(\tau, \sigma)$ vary in the interval $[0,1]$.    We note that  ($\tau=0$ or $\sigma=0$) and ($\tau=1$ or $\sigma=1$) the dampings are called frictional and viscous,  respectively.  Our main contribution is to show that the corresponding semigroup $S(t)=e^{\mathcal{B}t}$,  is analytic for $(\tau,\sigma)\in R_A:=[1/2,1]\times[ 1/2,1]$ and determine the Gevrey's class $\nu>\dfrac{1}{\phi}$, where $\phi=\left\{\begin{array}{ccc}
\dfrac{2\sigma}{\sigma+1} &{\rm for} & \sigma\leq \tau,\\\\
\dfrac{2\tau}{\tau+1} &{\rm for} &  \tau\leq \sigma.
\end{array}\right.$ \quad and \quad   $(\tau,\sigma)\in R_{CG}:= (0,1)^2$.
\end{abstract}


\setcounter{equation}{0}

\section{Introduction}

In this paper we study the regularity  properties of the  Timoshenko system  \cite{Timoshenko-1937}:
\begin{eqnarray*}
\rho {\bf A} u_{tt}-S_x=0&= & 0, \quad  x\in (0,L), \;t\in \mathbb{R}^+ ,\\
\rho I\psi_{tt}-M_x+S&= & 0, \quad  x\in (0,L), \;t\in \mathbb{R}^+.
\end{eqnarray*}
The constant $\rho$ denotes the density,   $ {\bf A}$ the cross-sectional area, and $I$ the area moment of inertia.  By $S$ we denote the shear force, and $M$ is the bending moment.  The function $u$ is the transverse displacement and $\psi$  is the rotation angle of a filament of the beam.  Here,  $t$ is the time variable, and $x$ is the space coordinate along the beam.  The constitutive  laws we use are the following:
\begin{equation*}
S= k {\bf A} G(u_x+\psi)\qquad {\rm and}\qquad M=EI\psi_x,
\end{equation*}
$E$ and $G$ are elastic constants, $k$  the shear coefficient for measuring the stiffness of materials ($k < 1$).

To simplify the notation let us denote by $\rho_1 =\rho  {\bf A}$,  $ \rho_2= \rho I$,  $\kappa = k {\bf A}G$,  $b = EI$ and taking  boundary conditions of type Dirichlet--Dirichlet
\begin{equation}\label{Eq08A}
u(t,0)=u(t,L)=0, \qquad \psi(t,0)=\psi(t,L)=0, \qquad  t>0, 
\end{equation}
using  the operator: $A:D(A)\subset L^2(0,L)\con L^2(0,L)$, where
\begin{equation}\label{Omenoslaplaciano}
A=-\Delta,\quad D(A)=H^2(0,L)\cap H^1_0(0,L).
\end{equation}
The system  in abstract form, is  rewrite, as
\begin{eqnarray}
\label{Eq05A}
\rho_1 u_{tt}-\kappa(u_x+\psi)_x+A^\tau u_t&= & 0, \quad  x\in (0,L), \;t\in \mathbb{R}^+ ,\\
\label{Eq06A}
\rho_2\psi_{tt}+bA\psi+\kappa(u_x+\psi)+A^\sigma \psi_t&=& 0, \quad  x\in (0,L), \;t\in \mathbb{R}^+.
\end{eqnarray}
and  initial conditions
\begin{equation}\label{Eq07A}
u(0,\cdot)=u_0, \quad u_t(0,\cdot)=u_1,\quad \psi(0,\cdot)=\psi_0\quad{\rm and}\quad \psi_t(0,\cdot)=\psi_1.
\end{equation}
 
In this paper we study the regularity of the $(u,\psi)$ solutions of the system \eqref{Eq05A}--\eqref{Eq07A}, 
where both parameters $\tau$ and $\sigma$ take values in the range $[0,1]$. 
It is known that the  this operator given in \eqref{Omenoslaplaciano} is selfadjoint, positive and has inverse compact on a complex Hilbert space $D(A^0)=L^2(0,L)$. Therefore,  the operator $A^{\theta}$ is self-adjoint positive for  all  $\theta\in\mathbb{R}$,  bounded for $\theta\leq 0$,  and  the embedding
\begin{eqnarray*}
D(A^{\theta_1})\hookrightarrow D(A^{\theta_2}),
\end{eqnarray*}
is continuous for $\theta_1>\theta_2$. Here, the norm in $D(A^{\theta})$ is given by $\|u\|_{D(A^{\theta})}:=\|A^{\theta}u\|$, $u\in D(A^{\theta})$, where $\|\cdot\|$ denotes the norm in the Hilbert space $\mathbb{X}$. Some of these spaces are: $D(A^{1/2})=H_0^1(0,L)$ and $D(A^{-1/2})=H^{-1}(0,L)$. 

During the past several decades, many authors have studied some physical phenomena for the Timoshenko system formulated into different mathematical models.  Let us mention some results of this Timoshenko's system. 

Kim and Renardy have already considered the case of two boundary control forces \cite{KR1987}(1987) for the Timoshenko beam. They proved the exponential decay of the energy by using a multiplier technique and provided numerical estimates of the eigenvalues of the operator associated with this system.  Shi and Feng \cite{SF2001}(2001) established the exponential decay of the energy with locally distributed feedback (two feedbacks).  Fatori et al. \cite{LFJMRR2014}(2014) studied the  thermoelastic beam system when the oscillations are defined by   Timoshenko's model    and the heat conduction is given by Green and Naghdi theories.  They showed that the corresponding semigroup is exponentially stable if and only if the wave velocities associated with the hyperbolic part of the system are equal.   In the case of lack of exponential stability, they show that the solution decays polynomially and that the  decay rate is optimal. Ammar-Khodja et al. \cite{AKABJEMRivera2003}(2003)  showed exponential decay when exponential kernels are considered,  while polynomial kernels are shown to lead to polynomial decay and prove that the polynomial rates found are optimal.

The primary motivation for choosing the Timoshenko model to study regularity (Gevrey class and Analyticity) came from two works. The first of the study of exponential decay in the Timoshenko system provided with two weak (frictional) dampings, published in Raposo et al. \cite{RFMN2005}(2005), whose model is given by:
\begin{eqnarray*}
\rho_1u_{tt}-k(u_x-\psi)_x+u_t=0\quad {\rm in}\quad (0,L)\times (0,\infty),\\
\rho_2\psi_{tt}-b\psi_{xx}+k(u_x-\psi)+\psi_t=0,\quad{\rm in}\quad (0,L)\times (0,\infty),\\
u(0,t)=u(L,t)=\psi(0,t)=\psi(L,t)=0\qquad t>0.
\end{eqnarray*}
The second most recent work was from 2016 \cite{AMJaime2016}, the authors 
considering  the stress-strain constitutive law is of Kelvin-Voigt type, given by
\begin{equation*}
S=k {\bf A} G(u_x+\psi)+\gamma_1(u_x+\psi)_t,\qquad M=EI\psi_x+\gamma_1\psi_{xt}.
\end{equation*}
Having the $\gamma_1, \gamma_2>0$,  they demonstrated the analyticity of the semigroup associated with the model. Already if one of the coefficients $\gamma_1$, $\gamma_2$ vanishes, then $S(t)$ is not exponentially stable but decays polynomially to zero, and  the polynomial decay rate is optimal. The studied system was:
\begin{eqnarray*}
\rho_1u_{tt}-k(u_x-\psi)_x-\gamma_1(u_x+\psi)_{xt}=0\quad {\rm in}\quad (0,L)\times (0,\infty),\\
\rho_2\psi_{tt}-b\psi_{xx}+k(u_x-\psi)+\gamma_1(u_x+\psi)_t-\gamma_2\psi_{xxt}=0,\quad{\rm in}\quad (0,L)\times (0,\infty),\\
u(0,t)=u(L,t)=\psi_x(0,t)=\psi_x(L,t)=0\qquad t>0.
\end{eqnarray*}

Both works applied a characterization of the theory of semigroups, together with  spectral analysis, both for the existence and for the exponential stability and analyticity of the semigroups. This theory of asymptotic behavior and regularity (existence of Gevrey class and nalyticity), was initially presented by Gearhart. \cite{Gearhart}(1978), Pr\"us\cite{Pruss}(1984) and later published in the book by Liu -Zheng( Theorem 1.3.2 )\cite{LiuZ}(1999).

The model investigated here is relevant from the mathematical point of view; it also has great importance in other sciences, such as mechanics. More precisely, the physical meaning of the presence of the two fractional dampings will guarantee that the beams that undergo small deformations can be quickly stabilized (with an exponential rate) so that the system with two dampings has at least Gevrey class solutions, the same as both. Damping is weak (frictional damping) and acts naturally, both on the small transverse vibrations $u$ and the beam filaments' angle of rotation $\psi$.

The Gevrey class semigroup has more regular properties than a differentiable semigroup,  but is less regular than an analytic semigroup. The Gevrey rate $\frac{1}{\nu} > 1$ `measures'  the divergence degree of its power series. It should be noted that the Gevrey class or analyticity of the model implies three essential  properties. The first is the property of the smoothing effect on the initial data; that is, no matter how irregular the initial data. The solutions of the model are very smooth in a finite time. The second property is that systems are exponentially stable.  Finally,  the systems enjoy the property of linear stability, which means that the type of the semigroup equals the spectral limit of its infinitesimal operator.

About investigations of the regularity of semigroups associated with various mathematical models, we can cite, for example,   Fatori et al.\cite{LMJAIME2012}(2012); in that work, the authors study the differentiability,  analyticity of the associated semigroup and also determine the optimum rate of decay.   More recently   published works explore the regularity of solutions using  Gevrey's class introduced in the  thesis of Taylor \cite{TaylorM}(1989). Regarding analyticity.  In the same book by Liu-Zheng, they also presented a theorem for analyticity. Among these works, we can mention Hao et al.\cite{Hao-2015}(2015).  Recently in 2023, the work \cite{FSLS2023} was published, in which the authors studied the abstract plate system:
\begin{equation*}
\left\{\begin{array}{c}
u_{tt}+\omega Au_{tt}+A^2u- A^\sigma\theta=0,\\
\theta_t+A\theta+A^\sigma u_t=0.
\end{array}\right.
\end{equation*}
Where the operator $A^\tau$ is selfadjoint and positive for all $\tau\in \mathbb{R}$ and  $\omega\geq 0$, for case $\omega=0$ and $A=-\Delta$, we have the thermoelastic plate of Euler-Bernoulli. For $\omega>0$ and $A=-\Delta$, we have the thermoelastic plate of Kirchhoff-Love; in this work is, determine the Gevrey sharp classes: for $\omega=0$; $s_{01}>\frac{1}{2\sigma-1}$ and $s_{02}>\sigma$, when $\sigma\in (1/2, 1)$ and $\sigma\in(1,3/2)$ respectively. Besides for $\omega>0$ we have $s_w>\frac{1}{4(\sigma-1)}$ when $\sigma\in(1,5/4)$. That work contains direct proofs of the analyticity of $S(t)$: In the case $\omega=0$,  we have analyticity for $\sigma=1$ and for the case $\omega>0$, $S(t)$ is analytic for $\sigma\in[5/4,3/2]$. In the same direction, we can cite recent research \cite{KAFSTebou2021, FMFSSB2023, GAJMRLiu2021}.

In the last decade, various investigations have emerged to study the stability and regularity of models with fractional damping. In this direction, we have, for example, the work of Sare et al.\cite{HSLiuRacke2019}; in that paper, the authors investigate thermoelastic-type coupled systems, where they address two cases with Fourier's heat law and the other with Cattaneo considering em both cases the rotational inertial term, study the exponential stability, possible regions of loss of exponential stability and polynomial stability, and, more recently, the work of Keyantuo et al.\cite{Tebou-2020}(2020) to be published. In this last work,  the authors studied the thermoelastic plate model with a fractional Laplacian between the Euler-Bernoulli and Kirchhoff model with two types of boundary conditions; in addition to studying the asymptotic and analytical behavior, the authors show that the underlying semigroups is of Gevrey's class $\delta$ for every $\delta>\frac{2-\theta}{2-4\theta}$ for  both the clamped and hinged boundary conditions when the parameter $\theta$ lies in the interval $(0,1/2)$.  At the same address, we can cite the investigations: \cite{RCRaposo2022,KZLZHugo2021,ZKLiuTebou2022,FSLS2023Book,BrunaJMR2022, Tebou-2021}. 

This paper is organized as follows. In section 2, we study the well-posedness of the system \eqref{Eq05A}-\eqref{Eq07A} through the semigroup theory.   We divide section 3 into two parts, subsection 3.1, devoted to the Gevrey class is, showing that the semigroup associated with the Tymoshenko system has a Gevrey class $\nu>\dfrac{1}{\phi}$, where $\phi=\left\{\begin{array}{ccc}
\dfrac{2\sigma}{\sigma+1} &{\rm for} & \sigma\leq \tau,\\\\
\dfrac{2\tau}{\tau+1} &{\rm for} & \tau\leq \sigma.
\end{array}\right.$ and    for the parameters $(\tau,\sigma)$ within the region $R_ {GC}$ and finally, in section 3.2, we address the analyticity of the semigroup associated with the system $S(t)=e^{\mathcal{B}t}$, we show that $S(t)$ is analytic when both two parameters $\tau$ and $\sigma $ take values in the interval $[1/2,1]$.

\section{Well-Posedness of the System}

We will use a semigroup approach to show existence uniqueness of  strong solutions for  laminated beams with  fractional  damping,  taking $w=u_t$, $v=\psi_t$,   and  considering $U=(u,w,\psi,v)$ and $U_0=(u_0,u_1,\psi_0, \psi_1)$,  the system \eqref{Eq05A}--\eqref{Eq07A},  can be written as an abstract Cauchy problem
\begin{equation}\label{Fabstrata}
    \frac{d}{dt}U(t)=\mathcal{B} U(t),\quad    U(0)=U_0,
\end{equation}
 where the operator $\mathcal{B}$ is given by
  \begin{equation} \label{operadorB}
 \mathcal{B}U:=\bigg (
 w,\;
 \frac{1}{\rho_1}[\kappa(u_x+\psi)_x- A^\tau w],\;
 v,  \; \frac{1}{\rho_2}[-bA\psi-\kappa(u_x+\psi)-A^\sigma v] \bigg),
  \end{equation}
for $U=(u,w,\psi, v)$.  This operator will be defined in a suitable subspace of the phase space obtained by performing duality products of the invariants $u_t$ and $\psi_t$ with the equations \eqref{Eq05A} and \eqref{Eq06A}, and using properties of the operators $A^\theta$ for $\theta\in \mathbb{R}$, defined by
$$              \begin{array}{ll}
               \mathcal{H}:= [D(A^\frac{1}{2})\times  D(A^0)]^2.
              \end{array}
$$

It's  a Hilbert space with the inner product
\begin{eqnarray*}
\langle U, U^* \rangle_ \mathcal{H}  &:= &\rho_1 \dual{w}{w^*}+\rho_2 \dual{v}{v^*}+ \kappa\dual{u_x+\psi} {u_x^*+\psi^*}+b\dual{\psi_x} {\psi_x^*},
\end{eqnarray*}
for $U=(u,w,\psi, v), U^*=(u^*,w^*,\psi^*, v^*)\in \mathcal{H}$, 
and induced norm
\begin{eqnarray*}
\|U\|^2_ \mathcal{H}&:= &\rho_1 \|w\|^2+\rho_2 \| v\|^2+ \kappa\|u_x+\psi\|^2+b\|A^\frac{1}{2}\psi\|^2.
\end{eqnarray*}

 In these conditions,  we define the domain of $\mathcal{B}$ as
\begin{equation}\label{dominioB}
    \mathcal{D}(\mathcal{B}):=D(A)\times \big(D(A^\frac{1}{2})\cap D(A^\tau)\big) \times D(A)\times \big(D(A^\frac{1}{2})\cap D(A^\sigma)\big).
    \end{equation}
   
To show that the operator $\mathcal{B}$ is the generator of a $C_0\!-\!{\rm semigroup}$  we invoke a result from Liu-Zheng' book.
\begin{theorem}[see Theorem 1.2.4 in \cite{LiuZ}] \label{TLiuZ}
Let $ \mathcal{B}$ be a linear operator with domain $\mathcal{D}(\mathcal{B})$ dense in a Hilbert space $\mathcal{H}$. If $ \mathcal{B}$ is dissipative and $0\in\rho( \mathcal{A})$, the resolvent set of $ \mathcal{B}$, then $ \mathcal{B}$ is the generator of a $C_0\!-\!{\rm semigroup}$ of contractions on $\mathcal{H}$.
\end{theorem}

Let us see that the operator $\mathcal{B}$ in \eqref{operadorB}  satisfies the conditions of this theorem. Clearly, we see that $\mathcal{D}(\mathcal{B})$ is dense in $\mathcal{H}$. Effecting the internal product of $\mathcal{B}U$ with $U$, we have
\begin{equation}\label{eqdissipative}
\text{Re}\dual{\mathcal{B}U}{U}=  -\dfrac{1}{\rho_1}\|A^\frac{\tau}{2}w\|^2-\dfrac{1}{\rho_2}\|A^\frac{\sigma}{2} v\|^2, \quad\forall\ U\in \mathcal{D}(\mathcal{B}),
\end{equation}
that is, the operator $\mathcal{B}$ is dissipative.

To complete the conditions of the above theorem, it remains to show that $0\in\rho(\mathcal{B})$. Let $F=(f^1,f^2,f^3,f^4)\in \mathcal{H}$, let us see that the stationary problem $ \mathcal{B}U=F$ has a solution $U=(u,w, \psi, v)$.  From the definition of the operator  $\mathcal{B}$ given in
\eqref{operadorB}, this system can be written as                     
\begin{align}
	w=f^1,\qquad& \quad\quad  \kappa (u_x+\psi)_x=\rho_1f^2+A^\tau f^1, \label{exist-10A}\\
	v=f^3,\qquad &  \quad\quad  -bA\psi-\kappa (u_x+\psi) =A^\sigma f^3+\rho_2f^4. \label{exist-20A}
\end{align}
Therefore, it is not difficult to see that there exist only one solution $u$ and $\psi$
of the system
\begin{multline}\label{Eliptico001}
 \kappa (u_x+\psi)_x=\rho_1f^2+A^\tau f^1\in D(A^0)\\
  bA\psi-\kappa (u_x+\psi) =A^\sigma f^3+\rho_2f^4\in D(A^0), 
\end{multline}
from where we have that
$$\|U\|_\mathcal{H}\leq C\|F\|_\mathcal{H},$$
wich in particular implies that $\|\mathcal{B}^{-1}F\|_\mathcal{H}\leq \|F\|_\mathcal{H}$, so we have  that   $0$ belongs to the resolvent set $\rho(\mathcal{B})$.  Consequently, from Theorem \ref{TLiuZ}  we have  $\mathcal{B}$ is the generator of a contractions semigroup.

As a consequence of the above Theorem\eqref{TLiuZ} we have
\begin{theorem}
Given $U_0\in\mathcal{H}$ there exists a unique weak solution $U$ to  the problem \eqref{Fabstrata} satisfying 
$$U\in C([0, +\infty), \mathcal{H}).$$
Futhermore, if $U_0\in  \mathcal{D}(\mathcal{B}^k), \; k\in\mathbb{N}$, then the solution $U$ of \eqref{Fabstrata} satisfies
$$U\in \bigcap_{j=0}^kC^{k-j}([0,+\infty),  \mathcal{D}(\mathcal{B}^j).$$
\end{theorem}
\begin{theorem}[Lions' Interpolation]\label{Lions-Landau-Kolmogorov}  Let $\alpha<\beta<\gamma$. The there exists a constant $L=L(\alpha,\beta,\gamma)$ such that
\begin{equation}\label{ILLK}
\|A^\beta u\|\leq L\|A^\alpha u\|^\frac{\gamma-\beta}{\gamma-\alpha}\cdot \|A^\gamma u\|^\frac{\beta-\alpha}{\gamma-\alpha}
\end{equation}
for every $u\in D(A^\gamma)$.
\end{theorem}
\begin{proof}
See  Theorem  5.34 \cite{EN2000}.
\end{proof}
\begin{theorem}[Hilla-Yosida]\label{THY} A linear (unbounded) operator $\mathcal{B}$ is the infinitesimal generator of a $C_0-$semigroup of contractions $S(t)$, $ t\geq 0$, if and only if\\
$(i)$ $\mathcal{B}$ is closed and $\overline{\mathcal{D}(\mathcal{B})}=\mathcal{H}$,\\
$(ii)$ the resolvent set $\rho(\mathcal{B})$ of $\mathcal{B}$ contains $\mathbb{R}^+$ and for every $\lambda>0$, 
\begin{equation*}
\|(\lambda I-\mathcal{B})^{-1}\|_{\mathcal{L}(\mathcal{H})}\leq\dfrac{1}{\lambda}.
\end{equation*}
\end{theorem}
\begin{proof}
See \cite{Pazy}.
\end{proof}
\section{Regularization Results}

In this section,  using the semigroup approach we show   the corresponding semigroup $S(t) = e^{\mathcal{B}t}$ of the system \eqref{Eq05A}--\eqref{Eq07A}
is of Gevrey class 
$\nu>\dfrac{1}{\phi}$  where $\phi=\left\{\begin{array}{ccc}
\dfrac{2\sigma}{\sigma+1} &{\rm for} & {\color{blue}\sigma\leq \tau},\\\\
\dfrac{2\tau}{\tau+1} &{\rm for} & {\color{blue}\tau\leq \sigma}.
\end{array}\right.$ and    $(\tau,\sigma)\in R_{CG}:= (0,1)^2$ and  $S(t)$ is analytic  in the region $R_A:=\{(\tau,\sigma) \in \mathbb{R}^2/ (\tau,\sigma)\in [\frac{1}{2},1]^2 \}$.

For the Gevrey class we use the characterization results presented in \cite{SCRT1990} (adapted from \cite{TaylorM}, Theorem 4, p. 153]).  And for the study of analyticity the main tool we use is the characterization of analytical semigroups due to Liu and Zheng (See book by Liu-Zheng - Theorem 1.3.3).

In what follows: $C$,  $C_\delta$ and $C_\varepsilon$ will denote  positive constants that assume different values in different places.

Next, we present two lemmas where two estimates are tested that are fundamental for the determination of the Gevrey class and the analytics of the associated semigroup $S(t)=e^{\mathcal{B}t}$.
\begin{lemma}
Let $\delta> 0$. There exists a constant $C_\delta > 0$ such that the solutions of \eqref{Eq05A}--\eqref{Eq07A}
for $|\lambda|\geq  \delta$  satisfy the inequality
\begin{equation}\label{EstimaEquivalenteExp} 
\|U\|_\mathcal{H}^2\leq C_\delta \|F\|_\mathcal{H}\|U\|_\mathcal{H} \quad\rm{for}\quad C_\delta>0.
\end{equation}
\end{lemma}
\begin{proof}
If $\lambda\in \R$ and $F=(f^1,f^2,f^3,f^4)\in \mathcal{H}$ then the solution $U=(u,w,\psi, v)\in\hbox{D}(\mathcal{B})$ the resolvent equation $(i\lambda I- \mathcal{B})U=F$ can be written in the form
\begin{eqnarray}
i\lambda u-w &=& f^1\quad\rm{in}\quad D(A^\frac{1}{2}),\label{esp-10}\\
i\lambda w-\frac{\kappa}{\rho_1}(u_x+\psi)_x+\frac{1}{\rho_1}A^\tau w&=& f^2\quad\rm{in}\quad D(A^0),\label{esp-20}\\
i\lambda  \psi-v &=&f^3\quad \rm{in}\quad D(A^\frac{1}{2}),\label{esp-30}\\
i\lambda v+\frac{b}{\rho_2}A\psi+\frac{\kappa}{\rho_2}(u_x+\psi)+\frac{1}{\rho_2}A^\sigma v &=& f^4 \quad\rm{in}\quad D(A^0)\label{esp-40}.
\end{eqnarray}

Using the fact that the operator is dissipative $ \mathcal{B}$,  we have 
\begin{equation}\label{dis-10}
\dfrac{1}{\rho_1}\|A^\frac{\tau}{2}w\|^2+\dfrac{1}{\rho_2}\|A^\frac{\sigma}{2}v\|^2 =\text{Re}\dual{(i\lambda -\mathcal{B})U}{U}_\mathcal{H}= \text{Re}\dual{F}{U}_\mathcal{H}\leq \|F\|_\mathcal{H}\|U\|_\mathcal{H}.
\end{equation}
On the other hand,  performing the duality product of \eqref{esp-20} for $\rho_1 u$,    and recalling that the operator $A^\theta$  for all $\theta\in\mathbb{R}$ is
self-adjoint,  we obtain
\begin{equation*}
\kappa\dual{(u_x+\psi)}{u_x}=\rho_1\|w\|^2+\rho_1\dual{w}{f^1} -i\lambda\|A^\frac{\tau}{2}u\|^2
+\dual{A^\frac{\tau}{2} f^1}{A^\frac{\tau}{2} u}
+\rho_1\dual{f^2}{u}
\end{equation*}
now performing the duality product of \eqref{esp-40} for $\rho_2 \psi$,   and recalling that the operator $A$  is
self-adjoint,  we obtain
\begin{equation*}
\kappa\dual{(u_x+\psi)}{\psi}+b\|A^\frac{1}{2}\psi\|^2=\rho_2\|v\|^2+\rho_2\dual{v}{f^3}-i\lambda\|A^\frac{\sigma}{2}\psi\|^2+\dual{A^\frac{\sigma}{2}f^3}{A^\frac{\sigma}{2}\psi}+\rho_2\dual{f^4}{\psi}.
\end{equation*}
Adding the last 2 equations, we have
\begin{multline}\label{Eq001Exponential}
\kappa\|u_x+\psi\|^2+b\|A^\frac{1}{2}\psi\|^2=
\rho_1\|w\|^2+\rho_2\|v\|^2
-i\lambda\{\|A^\frac{\tau}{2}u\|^2+\|A^\frac{\sigma}{2}\psi\|^2\}
 \\+\rho_1\dual{w}{f^1}+\rho_2\dual{v}{f^3}
 +\dual{A^\frac{\sigma}{2}f^3}{A^\frac{\sigma}{2}\psi}
+\rho_2\dual{f^4}{\psi}\\
+\dual{A^\frac{\tau}{2} f^1}{A^\frac{\tau}{2} u}
+\rho_1\dual{f^2}{u}.
\end{multline}
Taking real part,  using norm $\|F\|_\mathcal{H}$ and $\|U\|_\mathcal{H}$ and     applying Cauchy-Schwarz and Young Inequalities,  obtain
\begin{equation*}
\kappa\|u_x+\psi\|^2+b\|A^\frac{1}{2}\psi\|^2 
\leq \rho_1\|w\|^2+\rho_2\|v\|^2+C_\delta\|F\|_\mathcal{H}\|U\|_\mathcal{H}.
\end{equation*}
From  estimative \eqref{dis-10}  and the fact $0\leq\frac{\tau}{2}$ and $0\leq\frac{\sigma}{2}$  the continuous embedding $D(A^{\theta_2}) \hookrightarrow D(A^{\theta_1}),\;\theta_2>\theta_1$,  we have
\begin{eqnarray}\label{Exponential002}
\kappa\|u_x+\psi\|^2+b\|A^\frac{1}{2}\psi\|^2&\leq& C_\delta\|F\|_\mathcal{H}\|U\|_\mathcal{H}.
\end{eqnarray}
Finally,  from estimates \eqref{dis-10} and \eqref{Exponential002},  we complete the proof of this lemma.
\end{proof}
\begin{lemma}\label{Regularidad}
Let $\delta> 0$. There exists a constant $C_\delta > 0$ such that the solutions of \eqref{Eq05A}--\eqref{Eq07A}
for $|\lambda|\geq  \delta$  satisfy the inequality
\begin{equation}\label{EqPrincipalLema}
|\lambda|[\kappa\|u_x+\psi\|^2+b\|A^\frac{1}{2}\psi\|^2] \leq
|\lambda|[\rho_1\|w\|^2+\rho_2\|v\|^2] +C_\delta\|F\|_\mathcal{H}\|U\|_\mathcal{H}.
\end{equation}
\end{lemma}
\begin{proof}
Performing the duality product of \eqref{esp-20} for $\rho_1\lambda u$,      and recalling that the operator $A$  is
self-adjoint,  we obtain
\begin{eqnarray*}
\kappa\lambda\dual{(u_x+\psi)}{u_x}=\rho_1\lambda\|w\|^2+\rho_1\dual{\lambda w}{f^1}-i\|A^\frac{\tau}{2}w\|^2-i\dual{A^\tau w}{f^1}\\
+i\rho_1\dual{f^2}{w}+i\rho_1\dual{f^2}{f^1}.
\end{eqnarray*}
Now performing the duality product of \eqref{esp-40} for $\rho_2\lambda \psi$,   and recalling that the operator $A$  is
self-adjoint,  we have 
\begin{eqnarray*}
\kappa\lambda\dual{(u_x+\psi)}{\psi}+
b\lambda\|A^\frac{1}{2}\psi\|^2
=\rho_2\lambda\|v\|^2+\rho_2\dual{\lambda v}{f^3}-i\|A^\frac{\sigma}{2}v\|^2-i\dual{A^\sigma v}{f^3}\\
+i\rho_2\dual{f^4}{v}+i\rho_2\dual{f^4}{f^3}.
\end{eqnarray*}
Adding the last 2 equations,  we have
\begin{multline}\label{Eq001Gevrey}
\lambda[\kappa\|u_x+\psi\|^2+b\|A^\frac{1}{2}\psi\|^2] =
\lambda[\rho_1\|w\|^2+\rho_2\|v\|^2]\\
-i\{\|A^\frac{\tau}{2}w\|^2+\|A^\frac{\sigma}{2}v\|^2\}+\rho_1\dual{\lambda w}{f^1}+\rho_2\dual{\lambda v}{f^3}-i\{ \dual{A^\tau w}{f^1}+\dual{A^\sigma v}{f^3} \}\\
+i\rho_1\dual{f^2}{w} +i\rho_2\dual{f^4}{v}+i\rho_1\dual{f^2}{f^1}+i\rho_2\dual{f^4}{f^3}.   
\end{multline}
On the other hand, from \eqref{esp-20}  and \eqref{esp-40}, we have
\begin{multline}\label{Eq002Gevrey}
\{\rho_1\dual{\lambda w}{f^1}+\rho_2\dual{\lambda v}{f^3}\}\\
 =i\{\kappa\dual{(u_x+\psi)}{f^1_x}+\dual{A^\tau w}{f^1}-\rho_1\dual{f^2}{f^1}\\
 +b\dual{A^\frac{1}{2}\psi}{A^\frac{1}{2}f^3}+\kappa\dual{(u_x+\psi)}{f^3}+\dual{A^\sigma v}{f^3}-\rho_2\dual{f^4}{f^3}\}.
\end{multline}
Using the identity \eqref{Eq002Gevrey} in the \eqref{Eq001Gevrey} equation and simplifying, we get
\begin{multline}\label{Eq003Gevrey}
\lambda[\kappa\|u_x+\psi\|^2+b\|A^\frac{1}{2}\psi\|^2] =
\lambda[\rho_1\|w\|^2+\rho_2\|v\|^2]
-i\{\|A^\frac{\tau}{2}w\|^2+\|A^\frac{\sigma}{2}v\|^2\}\\
+i\rho_1\dual{f^2}{w}+i\rho_2\dual{f^4}{v}+i\kappa\dual{u_x}{f^1_x}+i\kappa\dual{\psi}{f^1_x}\\
+i\kappa\dual{u_x}{f^3}+i\kappa\dual{\psi}{f^3}+ib\dual{A^\frac{1}{2}\psi}{A^\frac{1}{2}f^3}.
\end{multline}
Taking the real part,  e applying the Cauchy-Schwarz and Young inequalities and using the definitions of the $F$ and $U$ norm,  we complete the proof of this lemma.
\end{proof}
\subsection{Gevrey's class}
\begin{definition}\label{Def1.1Tebou} Let $t_0\geq 0$ be a real number. A strongly continuous semigroup $S(t)$, defined on a Banach space $ \mathcal{H}$, is of Gevrey class $\nu > 1$ for $t > t_0$, if $S(t)$ is infinitely differentiable for $t > t_0$, and for every compact set $K \subset (t_0,\infty)$ and each $\mu > 0$, there exists a constant $ C = C(\mu, K) > 0$ such that
	\begin{equation}\label{DesigDef1.1}
	||S^{(n)}(t)||_{\mathcal{L}( \mathcal{H})} \leq  C\mu ^n(n!)^\nu,  \text{ for all } \quad t \in K, n = 0,1,2...
	\end{equation}
\end{definition}
\begin{theorem}[\cite{TaylorM}]\label{Theorem1.2Tebon}
	Let $S(t)$  be a strongly continuous and bounded semigroup on a Hilbert space $ \mathcal{H}$. Suppose that the infinitesimal generator $\mathcal{B}$ of the semigroup $S(t)$ satisfies the following estimate, for some $0 < \phi < 1$:
	\begin{equation}\label{Eq1.5Tebon2020}
	\lim\limits_{|\lambda|\to\infty} \sup |\lambda |^\phi ||(i\lambda I-\mathcal{B})^{-1}||_{\mathcal{L}( \mathcal{H})} < \infty. 
	\end{equation}
	Then $S(t)$  is of Gevrey  class  $\nu$   for $t>0$, for every   $\nu >\dfrac{1}{\phi}$.
\end{theorem}
Our main result in this subsection is as follows:
\begin{theorem} \label{GevreyLaminado} The  semigroup  $S(t)=e^{\mathcal{B}t}$  associated to system \eqref{Eq05A}--\eqref{Eq07A}  is of Gevrey class $\nu$ for every $\nu>\dfrac{1}{\phi}$  for
\begin{equation*}
\phi=\left\{ \begin{array}{ccc}
\dfrac{2\sigma}{\sigma+1} & {\rm for} & \sigma\leq\tau, \\\\
\dfrac{2\tau}{\tau+1}  &{\rm for}  &  \tau\leq \sigma.
\end{array}\right.\hspace{2cm}(\tau,\sigma)\in R_{CG}:=(0,1)^2.
\end{equation*} 
\end{theorem}
\begin{proof}
From the resolvent  equation $F=(i\lambda I-\mathcal{B})U$  for $\lambda\in \mathbb{R}$,  we have
$U=(i\lambda I-\mathcal{B})^{-1}F$.  Furthermore to show  \eqref{Eq1.5Tebon2020} this is   Theorem \ref{Theorem1.2Tebon} it is enough to show:
\begin{equation}\label{EstimaEquivalenteGevrey}
\dfrac{|\lambda|^\phi \|(i\lambda-\mathcal{B})^{-1}F\|_\mathcal{H}}{\|F\|_\mathcal{H}}\leq C  \Longleftrightarrow |\lambda |^\phi\|U\|_\mathcal{H}^2\leq C_\delta\|F\|_\mathcal{H}\|U\|_\mathcal{H},
\end{equation}
$\quad\rm{for}\quad \phi=\left\{ \begin{array}{ccc}
\dfrac{2\sigma}{\sigma+1} & {\rm for} & \sigma\leq\tau,\\\\
\dfrac{2\tau}{\tau+1}  &{\rm for}  & \tau\leq \sigma. \quad
\end{array}\right.\quad {\rm and}\quad (\tau,\sigma)\in R_{CG}:=(0,1)^2$.

Next, we will estimate $|\lambda|^\frac{2\tau}{\tau+1}\|w\|^2$ and $|\lambda|^\frac{2\sigma}{\sigma+1}\|v\|^2$.  \\\\
Let's start by estimating the term $|\lambda|^\frac{2\tau}{\tau+1}\|w\|^2$.  We assume $\lambda\in\R$ with  $|\lambda|>1$, we shall borrow some ideas from \cite{LiuR95}.  Set $w=w_1+w_2$, where $w_1\in D(A)$ and $w_2\in D(A^0)$, with 
\begin{multline}\label{Eq110AnalyRR}
i\lambda w_1+A w_1=f^2, 
\hspace{1cm} i\lambda w_2=-\dfrac{\kappa}{\rho_1}Au+\dfrac{\kappa}{\rho_1}\psi_x-\dfrac{1}{\rho_1}A^\tau w+Aw_1.
\end{multline} 
Firstly,  applying the product duality  the first equation in \eqref{Eq110AnalyRR} by $w_1$,    then by $Aw_1$ and recalling that the operator $A$  is
self-adjoint, we have 
\begin{equation}\label{Eq112AnalyRR}
|\lambda|\|w_1\|+|\lambda|^\frac{1}{2}\|A^\frac{1}{2}w_1\|+\|Aw_1\|\leq C\|F\|_\mathcal{H}.
\end{equation}
In follows from the second equation in \eqref{Eq110AnalyRR} that
\begin{equation*}
i\lambda A^{-\frac{1}{2}}w_2= -\dfrac{\kappa}{\rho_1}A^\frac{1}{2}u+\dfrac{\kappa}{\rho_1}A^{-\frac{1}{2}}\psi_x-\dfrac{1}{\rho_1}A^{\tau-\frac{1}{2}} w+A^\frac{1}{2}w_1,
\end{equation*}
 then,  as $\|A^{-\frac{1}{2}}\psi_x\|^2=\dual{-A^{-{\frac{1}{2}}}\psi_{xx}}{A^{-\frac{1}{2}}\psi}=\|\psi\|^2$,   $\tau-\frac{1}{2}\leq \frac{\tau}{2}$ and  $0\leq \frac{1}{2}$,  taking into account the continuous embedding $D(A^{\theta_2}) \hookrightarrow D(A^{\theta_1}),\;\theta_2>\theta_1$,  we have
\begin{equation}\label{Eq113AAnaly}
|\lambda|^2 \|A^{-\frac{1}{2}} w_2\|^2\leq C\{ \|A^\frac{1}{2}u\|^2+\|A^\frac{1}{2}\psi\|^2+\|A^\frac{\tau}{2}w\|^2\}+\|A^\frac{1}{2}w_1\|^2.
\end{equation}
Using the exponential decay estimates \eqref{EstimaEquivalenteExp} and   estimative  \eqref{Eq112AnalyRR},  considering  $|\lambda|>1$ as $-1<-\frac{2\tau}{\tau+1}$, we have
\begin{multline*}
|\lambda|^2\|A^{-\frac{1}{2}}w_2\|^2 \leq   C\|F\|_\mathcal{H}\|U\|_\mathcal{H}+|\lambda|^{-1}\|F\|^2_\mathcal{H}\\
\leq   C|\lambda|^{-\frac{2\tau}{\tau+1}}\{|\lambda|^\frac{2\tau}{\tau+1}\|F\|_\mathcal{H}\|U\|_\mathcal{H}+\|F\|^2_\mathcal{H}\}.
\end{multline*}
Then
\begin{equation}
\label{Eq113AnalyRR}
\|A^{-\frac{1}{2}}w_2\|^2\leq C_\delta |\lambda|^{-\frac{4\tau+2}{\tau+1}}\{|\lambda|^\frac{2\tau}{\tau+1}\|F\|_\mathcal{H}\|U\|_\mathcal{H}+\|F\|^2_\mathcal{H}\}.
\end{equation}
On the  other hand, from $w_2=w-w_1$,  \eqref{dis-10} and  as $\frac{\tau}{2}\leq\frac{1}{2}$ the inequality  \eqref{Eq112AnalyRR}, we  have 
\begin{equation}\label{Eq114AnalyRR}
\|A^\frac{\tau}{2} w_2\|^2\leq  \|A^\frac{\tau}{2} w\|^2+\|A^\frac{\tau}{2}w_1\|^2
\leq  C|\lambda|^{-\frac{2\tau}{\tau+1}}\{|\lambda|^\frac{2\tau}{\tau+1}\|F\|_\mathcal{H}\|U\|_\mathcal{H}+\|F\|^2_\mathcal{H}\}.
\end{equation}
Now,  by Lions' interpolations inequality for $0\in [-\frac{1}{2}, \frac{\tau}{2}]$, we derive
\begin{equation}\label{Eq115AnalyRR}
 \|w_2\|^2 \leq C(\|A^{-\frac{1}{2}}w_2\|^2)^\frac{\tau}{\tau+1}(\|A^\frac{\tau}{2}w_2\|^2)^\frac{1}{\tau+1}.
\end{equation}
Using \eqref{Eq113AnalyRR} and  \eqref{Eq114AnalyRR} in \eqref{Eq115AnalyRR}, we have
\begin{equation}\label{Eq118AnalyRR}
\|w_2\|^2 \leq C_\delta |\lambda|^{-\frac{4\tau}{\tau+1}}\{|\lambda|^\frac{2\tau}{\tau+1}\|F\|_\mathcal{H}\|U\|_\mathcal{H}+\|F\|^2_\mathcal{H}\}\quad{\rm for }\quad 0\leq\tau\leq 1.
\end{equation}
Now,   as $|\lambda|\|w\|^2\leq |\lambda| \|w_1\|^2+ |\lambda| \|w_2\|^2$, estimates \eqref{Eq112AnalyRR} and \eqref{Eq118AnalyRR} and as $-2\leq -\frac{4\tau}{\tau+1}$, we get
\begin{equation}\label{Eq119AnalyRR}
|\lambda|\|w\|^2\leq C|\lambda|^\frac{1-3\tau}{\tau+1}\{ |\lambda|^\frac{2\tau}{\tau+1}\|F\|_\mathcal{H}\|U\|_\mathcal{H}+\|F\|^2_\mathcal{H}\}\qquad\rm{for}\qquad 0\leq\tau\leq 1.
\end{equation}
Finally,  let's now estimate the missing term  $|\lambda|\|v\|^2$,  we assume  $|\lambda|>1$.  Set $v=v_1+v_2$, where $v_1\in D(A)$ and $v_2\in D(A^0)$, with 
\begin{multline}\label{Eq110AnalyRRW}
i\lambda v_1+A v_1=f^4 
\hspace{1cm} i\lambda v_2=-\dfrac{b}{\rho_2}A\psi -\dfrac{\kappa}{\rho_2}u_x-\dfrac{\kappa}{\rho_2}\psi-\dfrac{1}{\rho_2}A^\sigma v+Av_1.
\end{multline} 
Firstly,  applying the product duality  the first equation in \eqref{Eq110AnalyRRW} by $v_1$,  then by $Av_1$ and recalling that the operator $A$ is self-adjoint, we have 
\begin{equation}\label{Eq112AnalyRRW}
|\lambda|\|v_1\|+|\lambda|^\frac{1}{2}\|A^\frac{1}{2}v_1\|+\|Av_1\| \leq C \|F\|_\mathcal{H}.
\end{equation}
 In follows from the second equation in \eqref{Eq110AnalyRRW},  that
\begin{equation*}
i\lambda A^{-\frac{1}{2}}v_2= -\dfrac{b}{\rho_2}A^\frac{1}{2}\psi-\dfrac{\kappa}{\rho_2}A^{-\frac{1}{2}}u_x-\dfrac{\kappa}{\rho_2}A^{-\frac{1}{2}}\psi-\dfrac{1}{\rho_2}A^{\sigma-\frac{1}{2}}v+A^\frac{1}{2}v_1,
\end{equation*}
 then,  as $\|A^{-\frac{1}{2}}u_x\|^2=\dual{-A^{-\frac{1}{2}}u_{xx}}{A^{-\frac{1}{2}}u}=\|u\|^2$,    $-\frac{1}{2}\leq \frac{1}{2}$  and $\sigma-\frac{1}{2}\leq\frac{\sigma}{2}\leq\frac{1}{2}$, taking into account the continuous embedding $D(A^{\theta_2}) \hookrightarrow D(A^{\theta_1}),\;\theta_2>\theta_1$ and using \eqref{Eq112AnalyRRW} and $-1\leq-\frac{2\sigma}{\sigma+1}$, we have
\begin{equation*}
|\lambda|^2 \|A^{-\frac{1}{2}} v_2\|^2 \leq 
 C_\delta |\lambda|^{-\frac{2\sigma}{\sigma+1}} \{|\lambda|^\frac{2\sigma}{\sigma+1}\|F\|_\mathcal{H}\|U\|_\mathcal{H}+\|F\|_\mathcal{H}^2\},
\end{equation*}
then
\begin{equation}
\label{Eq113AAnalyRRW}
 \|A^{-\frac{1}{2}} v_2\|^2\leq  C_\delta |\lambda|^{-\frac{4\sigma+2}{\sigma+1}} \{|\lambda|^\frac{2\sigma}{\sigma+1}\|F\|_\mathcal{H}\|U\|_\mathcal{H}+\|F\|_\mathcal{H}^2\}\quad {\rm for } \quad 0\leq\sigma\leq 1.
\end{equation}
On the  other hand, from $v_2=v-v_1$,  \eqref{dis-10} and  as $\frac{\sigma}{2}\leq\frac{1}{2}$ the inequality  \eqref{Eq112AnalyRRW}, we  have 
\begin{equation}\label{Eq114AnalyRRW}
\|A^\frac{\sigma}{2} v_2\|^2\leq  C\{\|A^\frac{\sigma}{2} v\|^2+\|A^\frac{\sigma}{2}v_1\|^2\}
\leq  C_\delta|\lambda|^{-\frac{2\sigma}{\sigma+1}}\{|\lambda|^\frac{2\sigma}{\sigma+1} \|F\|_\mathcal{H}\|U\|_\mathcal{H}+\|F\|^2_\mathcal{H}\}.
\end{equation}
Now,  by Lions' interpolations inequality Theorem \ref{Lions-Landau-Kolmogorov} for $0\in [-\frac{1}{2},\frac{\sigma}{2}]$, we derive
\begin{equation}\label{Eq115AnalyRRW}
 \|v_2\|^2\leq C(\|A^{-\frac{1}{2}}v_2\|^2)^\frac{\sigma}{\sigma+1} (\|A^\frac{\sigma}{2}v_2\|^2)^\frac{1}{\sigma+1}.
\end{equation}
Using   estimates \eqref{Eq113AAnalyRRW} and \eqref{Eq114AnalyRRW} in \eqref{Eq115AnalyRRW}, we have
\begin{equation}\label{Eq118AnalyRRW}
\|v_2\|^2\leq C|\lambda|^{-\frac{4\sigma}{\sigma+1}}\{|\lambda|^\frac{2\sigma}{\sigma+1} \|F\|_\mathcal{H}\|U\|_\mathcal{H}+\|F\|^2_\mathcal{H}\}.
\end{equation}
On the other hand,   as $|\lambda|\|v\|^2\leq C\{|\lambda| \|v_1\|^2+ |\lambda| \|v_2\|^2\}$,  from inequality \eqref{Eq112AnalyRRW} ($|\lambda|\|v_1\|^2\leq |\lambda|^{-1}\|F\|^2_\mathcal{H}$), \eqref{Eq118AnalyRRW} and as $-2\leq -\frac{4\sigma}{\sigma+1}$, we have
\begin{equation}\label{Eq119AnalyRRW}
|\lambda|\|v\|^2\leq C|\lambda|^{1-2\sigma}\{ |\lambda|^\sigma \|F\|_\mathcal{H}\|U\|_\mathcal{H}+\|F\|^2_\mathcal{H}\}\quad\rm{for}\quad 0\leq\sigma\leq 1.
\end{equation}
Finally using the estimates  \eqref{Eq119AnalyRR}   and \eqref{Eq119AnalyRRW}  in the inequality of Lemma \ref{Regularidad},   we arrive at the
\begin{multline}\label{Eq020Gevrey}
|\lambda|[\kappa\|u_x+\psi\|^2+b\|A^\frac{1}{2}\psi\|^2]
 \leq
C\bigg[|\lambda|^\frac{1-3\tau}{\tau+1} \{|\lambda|^\frac{2\tau}{\tau+1}\|F\|_\mathcal{H}\|U\|_\mathcal{H}+\|F\|^2_\mathcal{H}\} \\
+|\lambda|^\frac{1-3\sigma}{\sigma+1}\big\{|\lambda|^\frac{2\sigma}{\sigma+1} \|F\|_\mathcal{H}\|U\|_\mathcal{H} +\|F\|^2_\mathcal{H}\big \}  \bigg]
+C_\delta\|F\|_\mathcal{H}\|U\|_\mathcal{H}. 
\end{multline}
Therefore
\begin{multline}\label{Eq021Gevrey}
|\lambda|  [\kappa\|u_x+\psi\|^2+b\|A^\frac{1}{2}\psi\|^2]
 \\
 \leq  C_\delta\left\{ \begin{array}{ccc}
(i)\quad |\lambda|^\frac{1-3\tau}{\tau+1}\{|\lambda|^\frac{2\tau}{\tau+1}\|F\|_\mathcal{H}\|U\|_\mathcal{H}+\|F\|^2_\mathcal{H}\}& {\rm for} & \tau\leq\sigma,\\\\
(ii)\quad  |\lambda|^\frac{1-3\sigma}{\sigma+1}\{|\lambda|^\frac{2\sigma}{\sigma+1} \|F\|_\mathcal{H}\|U\|_\mathcal{H}+\|F\|^2_\mathcal{H}\}& {\rm for} &  \sigma\leq\tau.
\end{array} \right.  
\end{multline}
Finally  from estimates\eqref{Eq119AnalyRR},  \eqref{Eq119AnalyRRW} and \eqref{Eq021Gevrey}. We finish the proof of this theorem.

\end{proof}
\subsection{Analyticity  of $S(t)=e^{\mathcal{B}t}$ for $(\tau,\sigma)\in \big[\frac{1}{2},  1\big]^2$}

\begin{theorem}[see \cite{LiuZ}]\label{LiuZAnaliticity}
	Let $S(t)=e^{\mathcal{B}t}$ be $C_0$-semigroup of contractions  on a Hilbert space $ \mathcal{H}$. Suppose that
	\begin{equation}\label{EixoImaginary}
	\rho(\mathcal{B})\supseteq\{ i\lambda/ \lambda\in \R \} 	\equiv i\R
	\end{equation}
	 Then $S(t)$ is analytic if and only if
	\begin{equation}\label{Analiticity}
	 \limsup\limits_{|\lambda|\to
		\infty}
	\|\lambda(i\lambda I-\mathcal{B})^{-1}\|_{\mathcal{L}( \mathcal{H})}<\infty
	\end{equation}
	holds.
\end{theorem}
Before proving the main result of this section, we will prove the following lemma.
\begin{lemma}\label{Lema001Analiticity}
Let $\delta> 0$. There exists a constant $C_\delta > 0$ such that the solutions of \eqref{Eq05A}--\eqref{Eq07A}
for $|\lambda|\geq  \delta$  satisfy the inequality
\begin{eqnarray}\label{Eq004Lema02A}
(i)\quad |\lambda|\|w\|^2\leq  C_\delta\|F\|_\mathcal{H}\|U\|_\mathcal{H}\qquad{\rm for}\qquad \frac{1}{2}\leq\tau\leq 1.\\
\label{Eq005Lema02A}
(ii)\quad |\lambda|\|v\|^2\leq  C_\delta\|F\|_\mathcal{H}\|U\|_\mathcal{H}\qquad{\rm for}\qquad \frac{1}{2}\leq\sigma\leq 1.
\end{eqnarray}
\end{lemma}
\begin{proof}
We will initially show the item $(i)$ of this lemma,  performing the duality product of \eqref{esp-20} for $\rho_1A^{-\tau}\lambda w$,  and recalling that the operator $A^\theta$ is self-adjoint for all $\theta\in\mathbb{R}$,  we obtain
\begin{eqnarray*}
\lambda\|w\|^2\hspace*{-0.3cm} &=& \kappa\dual{\lambda(-Au+\psi_x)}{A^{-\tau} w}+\rho_1\dual{f^2}{\lambda A^{-\tau}w}-i\rho_1\lambda^2\|A^{-\frac{\tau}{2}}w\|^2\\
&=&i\kappa\dual{(Aw+Af^1)}{A^{-\tau}w} -i\kappa\dual{v_x+f^3_x}{A^{-\tau}w}\\
& & +\dual{f^2}{i\kappa A^{1-\tau}u-i\kappa A^{-\tau}\psi_x+iw-i\rho_1 A^{-\tau}f_2}-i\rho_1\lambda^2\|A^{-\frac{\theta}{2}}w\|^2\\
&=&\hspace*{-0.3cm}   i\kappa\|A^\frac{1-\tau}{2} w\|^2+i\kappa\dual{A^\frac{1}{2}f^1}{A^{\frac{1}{2}-\tau}w}+i\kappa\dual{v}{A^{-\tau}w_x}-i\kappa\dual{f_x^3}{A^{-\tau}w}-i\dual{f^2}{w}\\
& & +i\kappa\dual{f^2}{A^{1-\tau}u} +i\kappa\dual{A^{-\tau}f^2}{\psi_x}  -i\rho_1\|A^{-\frac{\tau}{2}}f^2\|^2-i\rho_1\lambda^2\|A^{-\frac{\theta}{2}}w\|^2.
\end{eqnarray*}
Noting that: $\|A^{-\tau} w_x\|^2=\dual{A^{-\tau}w_x}{A^{-\tau}w_x}=\dual{-A^{-\tau}w_{xx}}{A^{-\tau}w}=\|A^\frac{1-2\tau}{2} w\|^2$,  taking real part  and considering that  $\frac{1}{2}\leq\tau\leq 1$  using estimative \eqref{dis-10} and using Cauchy-Schwarz and Young  inequalities,  for $\varepsilon>0$ exists $C_\varepsilon$  such that
\begin{eqnarray*}
|\lambda|\|w\|^2 & \leq & C_\delta\|F\|_\mathcal{H}\|U\|_\mathcal{H}+\varepsilon\|A^{-\tau}w\|^2+C_\varepsilon\|v\|^2.
\end{eqnarray*}
As $0\leq\frac{\sigma}{2}$,  then from  estimative  \eqref{dis-10}  $\|v\|^2\leq C_\delta\|F\|_\mathcal{H}\|U\|_\mathcal{H}$.  From the continuous embedding for  $|\lambda|\geq 1$,  we finish the proof of item $(i)$ of this lemma.

 Again similarly,  performing the duality product of \eqref{esp-40} for $\rho_2A^{-\sigma}\lambda v$, using \eqref{esp-30}, and recalling the                                        self-adjointness of $A^\theta$,   $\theta \in\mathbb{R}$,  we obtain 
\begin{eqnarray*}
\lambda\|v\|^2 &=&   -b\dual{A\lambda \psi}{A^{-\sigma}v}-\kappa\dual{\lambda(u_x+\psi)}{A^{-\sigma} v} +\rho_2\dual{f^4}{\lambda A^{-\sigma}v}-i\lambda^2\|A^{-\frac{\sigma}{2}}v\|^2\\
&=& ib\|A^\frac{1-\sigma}{2}v\|^2+ib\dual{A^\frac{1}{2}f^3}{A^{\frac{1}{2}-\sigma}v}-i\kappa\dual{A^{-\sigma} w_x}{v}-i\kappa\dual{f^1_x}{A^{-\sigma} v}\\
& & +i\kappa\|A^{-\frac{\sigma}{2}}v\|^2+i\kappa\dual{f^3}{A^{-\sigma}v}-ib\dual{f^4}{A^{1-\sigma}\psi}-i\kappa\dual{f^4}{A^{-\sigma} u_x}\\
& & -i\kappa\dual{f^4}{A^{-\sigma}\psi}-i\dual{f^4}{v} +   i\rho_2\|A^{-\frac{\sigma}{2}}f^4\|^2-i\lambda^2\|A^{-\frac{\sigma}{2}}{v}\|^2.
\end{eqnarray*}
Noting that: $\|A^\frac{1-2\sigma}{2}w\|^2=\|A^{-\sigma} w_x\|^2$, taking real part  and considering that  $\frac{1}{2}\leq\sigma\leq 1$  and using Cauchy-Schwarz and Young  inequalities,  for $\varepsilon>0$ exists $C_\varepsilon$  such that
\begin{eqnarray*}
|\lambda|\|v\|^2 & \leq & C_\delta\|F\|_\mathcal{H}\|U\|_\mathcal{H}+C_\varepsilon\|A^\frac{1-2\sigma}{2}w\|^2+\varepsilon\|v\|^2.
\end{eqnarray*}
As $-\frac{1-2\sigma}{2} \leq 0\Longleftrightarrow \frac{1}{2}\leq \sigma\leq 1$.   Considering  $|\lambda|\geq 1$ and using  estimative  \eqref{EstimaEquivalenteExp},   we finish the proof of item $(ii)$ of this lemma.

\end{proof}

The main result of this subsection is the following theorem
\begin{theorem}
The semigroup $S(t)=e^{\mathcal{B}t}$ associated to the system \eqref{Eq05A}--\eqref{Eq07A} is analytic when both parameters $\tau$ and $\sigma$ vary by the interval $[\frac{1}{2},1]$.
\end{theorem}
\begin{proof}
We will prove this theorem using Theorem \eqref{LiuZAnaliticity},  so we must prove the conditions \eqref{EixoImaginary} and \eqref{Analiticity}.

Let us first check the condition \eqref{EixoImaginary} the Theorem\eqref{LiuZAnaliticity}. It'is  prove that $i\R\subset\rho( \mathcal{B})$ by contradiction.  
Since $\mathcal{B}$   is the infinitesimal generator of a $C_0-$semigroup of contractions $S(t)$, $ t\geq 0$,  from Theorem \ref{THY},  $\mathcal{B}$ is a closed operator,  and $\mathcal{D}(\mathcal{B})$ has compact embedding into the energy space $\mathcal{H}$, the spectrum $\sigma(\mathcal{B})$ contains only eigenvalues. Suppose that $i\R\not\subset \rho( \mathcal{B})$. As $0\in\rho( \mathcal{B})$ and  $\rho( \mathcal{B})$ is open, we consider the highest positive number $\lambda_0$ such that the interval  $(-i\lambda_0,i\lambda_0)\subset\rho( \mathcal{B})$ then $i\lambda_0$ or $-i\lambda_0$ is an element of the spectrum $\sigma( \mathcal{B})$. We Suppose $i\lambda_0\in \sigma( \mathcal{B})$ (if $-i\lambda_0\in \sigma( \mathcal{B})$ the proceeding is similar). Then, for $0<\nu<\lambda_0$ there exist a sequence of real numbers $(\lambda_n)$, with $0<\nu\leq\lambda_n<\lambda_0$, $\lambda_n\to \lambda_0$, and a vector sequence  $U_n=(u_n,w_n,\psi_n,v_n)\in \mathcal{D}( \mathcal{B})$ with  unitary norms, such that
\begin{eqnarray*}
\|(i\lambda_n- \mathcal{B}) U_n\|_\mathcal{H}=\|F_n\|_\mathcal{H}\to 0,
\end{eqnarray*}
as $n\to \infty$. From \eqref{Exponential002} for $0\leq\tau\leq 1$ and $0\leq\sigma\leq 1$, we have 
\begin{eqnarray*}
\kappa\|u_{xn}+\psi_n\|^2+b\|A^\frac{1}{2}\psi_n\|^2&\leq& C_\delta\|F_n\|_\mathcal{H}\|U_n\|_\mathcal{H}.
\end{eqnarray*}
In addition to the estimative \eqref{dis-10} for $0\leq\tau\leq 1$ and $0\leq\sigma\leq 1$,  we have 
\begin{eqnarray*}
\dfrac{1}{\rho_1}\|A^\frac{\tau}{2}w_n\|^2+\dfrac{1}{\rho_2}\|A^\frac{\sigma}{2}v_n\|^2\leq \|F_n\|_\mathcal{H}\|U_n\|_\mathcal{H}.
\end{eqnarray*}
Consequently,   $\|U_n\|^2 \to 0.$
Therefore,  we have  $\|U_n\|_\mathcal{H}\to  0$ but this is absurd,  since $\|U_n\|_\mathcal{H}=1$ for all $n\in\N$. Thus, $i\R\subset \rho(\mathcal{B})$. This completes the proof of condition \eqref{EixoImaginary} of the Theorem\eqref{LiuZAnaliticity}.
 
 Finally let's prove the condition \eqref{Analiticity}, note that proving this condition is equivalent to showing, let $\delta>0$. There exists a constant $C_\delta > 0$ such that the solutions of \eqref{Eq05A}--\eqref{Eq07A}
for $|\lambda|\geq  \delta$  satisfy the inequality
 \begin{equation}\label{EquivAnaliticity}
 |\lambda|\|U\|^2_\mathcal{H}\leq C_\delta\|F\|_\mathcal{H}\|U\|_\mathcal{H}.
 \end{equation}
 It is not difficult to see that this inequality \eqref{EquivAnaliticity} follows from the inequalities of the Lemmas \ref{Regularidad} and \ref{Lema001Analiticity},    so the proof of this theorem is finished.
\end{proof}


\end{document}